\documentclass[10pt]{amsart}
\usepackage{amssymb,amsmath,amsthm,amsfonts,amsopn,url}
\usepackage{amscd,amssymb,amsopn,amsmath,amsthm,graphics,amsfonts,enumerate,verbatim,calc}
\usepackage[all]{xy}
\theoremstyle{plain}
\newtheorem{thm}{Theorem}[section]
\newtheorem*{mt*}{Theorem A}
\newtheorem*{mt**}{Theorem B}
\newtheorem*{cj*}{Conjecture}
\newtheorem*{nt*}{Notations}

\newtheorem{lemma}[thm]{Lemma}
\newtheorem{cor}{Corollary}
\newtheorem{rem}{Remark}
\newtheorem{example}{Example}[section]
\theoremstyle{definition}

\newcommand{\func}[1]{\mathrm{#1} \,}

\newcommand{\supp}{\func{supp}}
\newcommand{\Star}{\func{Star}}
\newcommand{\Sd}{\func{Sd}}
\newcommand{\lk}{\func{lk}}
\newcommand{\del}{\func{del}}

\newcommand{\ZZ}{{\mathbb Z}}

\title
{The second vanishing theorem in Stanley-Reisner ring with topological interpretation}

\author[]{Rajsekhar Bhattacharyya}

\address{Dinabandhu Andrews College, Garia, Kolkata 700084, India}

\email{rbhattacharyya@gmail.com}




\begin{document}

\begin{abstract}
For regular local ring, the ``second vanishing theorem'' or ``SVT'' of local cohomology has been proved in several cases. In this paper, we explore the result similar to that of the SVT to Stanley-Reisner ring with an interpretation from combinatorial topology. 
\end{abstract}

\maketitle

\section{Introduction}

For Noetherian ring $S$ and for an ideal $J\subset S$, we have the local cohomology module $H^i_J(S)$ supported at $J$. It is a very mysterious object, even it is quite hard to know when it will vanish. Recently there have been several instances where more explicit information on local cohomology modules was obtained in special cases using combinatorial approach and one of such place is Stanley-Reisner ring.

Recall that the cohomological dimension of an ideal $J$ of a Noetherian ring $S$ is the maximum index $i \geq0$ for which the local cohomology module $H^i_J(S)$ is nonzero. In this context we mention Hartshorne-Lichtenbaum vanishing theorem or ``HLVT'' \cite{HartshorneCD}. It states that: For any complete local domain $S$ of dimension $d$, $H^d_J(S)$ vanishes if and only if $\dim(S/J)>0$. One may regard the HLVT as the ``first vanishing theorem'' for local cohomology. In \cite{RWY}, HLVT has been extended for Stanley-Reisner ring with an interpretation from combinatorial topology.

If the ring $S$ contains a field, the ``second vanishing theorem'' or ``SVT'' of local cohomology states the following: Let $S$ be a complete regular local ring of dimension $d$ with a separably closed residue field, which it contains. Let $J\subseteq S$ be an ideal such that $\dim(S/J)\geq 2$. Then $H^{d-1}_J(S)=0$  if and only if the punctured spectrum of $S/J$ is connected \cite{HartshorneCD,H-L,Ogus,P-S}. In \cite{CohDim}, the SVT has been extended to complete unramified regular local ring of mixed characteristic. In \cite{Bh}, it has been realized in the ramified case only the for extended ideals.

In this paper, we explore the result similar to that of the SVT to Stanley-Reisner ring with a combinatorial topological interpretetion, (see Theorem 3.2. and Theorem 3.3): 


In this context it should be mentioned that first part of the proof of Lemma 3.1 goes similarly to a part of the proof of Theorem 3.5 of \cite{RWY}. 


\section{Preliminaries}

In this section, we recall some basic results from combinatorial topology. For general references we refer \cite{Bj} and \cite{Mu}.

For $[n]= \{1,2,\ldots,n\}$, let $\Delta\subset 2^{[n]}$ be a simplicial complex i.e., if $F\in \Delta$ and $G\subset F$, then $G\in \Delta$. For $\bold{a}\in \ZZ^n$, we define the following support subsets of $[n]= \{1,2,\ldots,n\}$: $\supp_{+}\bold{a}= \{i: a_i > 0\}$, $\supp_{-}\bold{a}= \{i: a_i < 0\}$ and $\supp \bold{a}= \{i: a_i \neq 0\}$. Let $S = k[x_1, \ldots, x_n]$ be a polynomial ring over a field $k$. If $F\in \Delta\subset 2^{[n]}$, then we write square-free monomial $x^{a_1}_1\ldots x^{a_n}_n= x^F$, when $F= \supp \bold{a}$.

For a simplicial complex $\Delta$, the Stanley-Reisner ideal of $\Delta$ is the square-free monomial ideal, $I_{\Delta}= \{x^F: F\notin \Delta\}$. We define the Stanley-Reisner ring as $S/I_{\Delta}= k[\Delta]$. Let $\Sigma\subset \Delta \subset 2^{[n]}$ be two simplicial complexes. Then we have Stanley-Reisner ideals $I_{\Delta}\subset I_{\Sigma}$. In $k[\Delta]$, we denote the image of $I_{\Sigma}$ by $J$. 

Given a face $F\in \Delta$, we can define three subcomplexes, called the star, deletion, and link of $F$ inside $\Delta$, as follows:
$\Star_{\Delta} (F) = \{G\in \Delta: G\cup F\in \Delta\}$, $\del_{\Delta} (F) = \{G\in \Delta: G\not\subset F\}$ and $\lk_{\Delta} (F) = \{G\in \Delta: G\cap F= \phi, G\cup F\in \Delta\}$

Now to prove the result of this paper, we need the following result, Theorem 3.2 of \cite{RWY}:

\begin{mt*}
Let $\Sigma\subset \Delta$ be simplicial complexes and let $\bold{a}\in \ZZ^{n}$, $F_{+}=\supp_{+}\bold{a}$ and $F_{-}=\supp_{-}\bold{a}$. Then $H^{i}_J (k[\Delta])_{\bold{a}}= \widetilde{H}^{i-1}(||\Star_{\Delta}(F_{+})||-||\Sigma||,||\del_{\Star_{\Delta}(F_{+})}(F_{-})||-||\Sigma||;k)$
where $\widetilde{H}^{i}(-,-;k)$ denotes $i-th$ singular relative reduced cohomology and $||\Delta||$
denotes the geometric realization of a simplicial complex $\Delta$.
\end{mt*}

Let $\Sd (\Delta)$ denote the barycentric subdivision of the simplicial complexes $\Delta$, see Section 15 of \cite{Mu} and also section 9 of \cite{Bj} for abstract simplicial complexes. Given a subcomplex $\Sigma\subset \Delta$, $\Sd (\Delta-\Sigma)$ is the subcomplex of barycentric subdivision of $\Sd (\Delta)$ whose vertices are not the barycentre of any face of $\Sigma$. 

We also need the following result, Lemma 4.7.27 of \cite{BLSWZ}.

\begin{mt**}
Let $\Delta$ be a simplicial complex and let $\Delta'$ and $\Sigma$ be two subcomplexes. Then the pair of spaces $(||\Delta||-||\Sigma||,||\Delta'||-||\Sigma||)$ is relatively homotopy equivalent to the pair $(\Sd (\Delta-\Sigma),\Sd (\Delta'-\Sigma))$
\end{mt**}

\section{main results}

Before proving our main result, we need the following Lemma. 

\begin{lemma}
For $\Sigma \subset \Delta$ and $\dim\Delta= d-1$. Suppose, every $(d-1)$-face of $\Delta$ contains a vertex of $\Sigma$ and its every $(d-2)$-face is a $(d-2)$-face of $\Sigma$ and also every $(d-2)$-facet of $\Delta$ contains a vertex of $\Sigma$, then we can remove every $(d-1)$-face and every $(d-2)$-face of $\Sd(\Delta-\Sigma)$. 
\end{lemma}

\begin{proof} 
Let $\Sd (\Delta-\Sigma)$ be the subcomplex of barycentric subdivision of $\Sd (\Delta)$ whose vertices are not the barycentre of any face of $\Sigma$.Denote $\Sd(\Delta-\Sigma)= X$ and we can use variable $z$ for $d-1, d-2$. At first, we assume the faces of $X$ are either a $d-1$-face or a $d-2$-facet. To remove a $z$-face of $X$, we adopt the way of elementary collapse \cite{Bj}, which states that, if we find a $z$-face which contains some $(z-1)$-face, but latter is not in any other $z$-face, we remove both of them. We claim that this process removes all those $z$-faces of $X$ as mentioned in the begining of the paragraph. 

If not, assume that after removing few $z$-faces, when we reach the simplicial complex $X'$ and there exists one $z$-face in $X'$ which can not be removed, i.e. there is a $z$-face in $X'$ such that all of its $(z-1)$-faces lying in some other $z$-face, so that no further collapse is possible. Let $\sigma$ be one of such $z$-faces in $X'$, and let $G$ be a $z$-face of $X$ in which it lies, so that the barycenter $b_G$ of $G$ is one of its vertex.
 .
Then $\lk_{X'}(b_G)$ is a subcomplex of $\lk_{\Sd (\Delta)}(b_G)$ which is the boundary of $\Sd G$ (since being a barycentre only for $F\subset \Sd G$, $\{b_G\}\cup F$ is a face of $\Sd (\Delta)$). Since $\sigma$ is in $X'$, $\dim (\lk_{X'}(b_G))= \dim (\lk_{\Sd (\Delta)}(b_G))$, and moreover any $(z-2)$-face $\tau$, $(z-1)$-face $\tau\cup \{b_G\}$ is in two $z$-faces (otherwise, we could remove the face $\sigma$). Thus we find $\lk_{X'}(b_G)= \lk_{\Sd (\Delta)}(b_G)$. But this is a contradiction, since $\lk_{\Sd (\Delta)}(b_G)$ should contain atleast one vertex of $\Sigma$, while $\lk_{X'}(b_G)$ should not contain any vertex of $\Sigma$.

Now consider those $d-2$-faces which are in some $d-1$-face. Now we can form $(d-2)$-face either by taking barycentre of the $(d-1)$-face along with other $(d-2)$ vertices which are from the collection of vertices of the $(d-1)$-face and barycentres of $(d-2)$-faces where there should be atleast one barycentre from a $(d-2)$-face (the situation that all the $(d-2)$ vertices are from the vertices of the $(d-1)$ face has been considered earlier in this proof) or by taking $(d-1)$ vertices from the collection of vertices of the $(d-1)$-face and barycentres of $(d-2)$-faces. Since every $(d-2)$-face is a $(d-2)$-face of $\Sigma$, the barycentre of each $(d-2)$-face is missing in $X$ and hence in each of the cases, similar situation arises as above. Thus we can remove all the $(d-2)$-faces. 
\end{proof}

\begin{rem}
The result of above lemma is true if we replace $\Delta$ by some of its subcomplex $\Omega$
\end{rem}

Now we state the main result of this paper.

\begin{thm}
Let $\Sigma \subset \Delta$ be simplicial complexes where $\dim k[\Delta]= d$ with $J=I_{\Sigma}$ (assume the situation and notation in section 2). If $H^{d-1}_J (k[\Delta])=0$ then 
every $(d-2)$-face contains one vertex of $\Sigma$. 
\end{thm}

\begin{proof}
For the forward direction, we assume that some $(d-2)$-face $F$ contains no vertex of $\Sigma$ and we want to prove that for some multidegree $\bold{a}$, $H^{d-1}_J (k[\Delta])_{\bold{a}}\neq 0$. We consider a multidegree $\bold{a}$ having $F_{+}=\supp_{+}\bold{a}= \phi$, $F_{-}=\supp_{-}\bold{a}= F$. From Theorem A given in Section 2, we get $H^{d-1}_J (k[\Delta])_{\bold{a}}= \widetilde{H}^{d-2}(||\Delta||-||\Sigma||,||\del_{\Delta}(F)||-||\Sigma||;k)$ where we use the fact that $\Star_{\Delta}(\phi)=\Delta$. Now if we remove $||\Delta||-||\Sigma||-||F||$, using "exision" we get $\widetilde{H}^{d-2}(||\Delta||-||\Sigma||,||\del_{\Delta}(F)||-||\Sigma||;k)= \widetilde{H}^{d-2}(||F||,||\partial F||;k)$. Finally, going to the quotient space $||F||/||\partial F||$, we get $\widetilde{H}^{d-2}(||F||,||\partial F||;k)= \widetilde{H}^{d-2}(||F||/||\partial F||;k)= \widetilde{H}^{d-2}(\SS^{d-2};k)=k$. Thus $H^{d-1}_J (k[\Delta])_{\bold{a}}\neq 0$.
\end{proof}

\begin{cor}
If $H^{d-1}_J (k[\Delta])=0$ then $H^{d}_J (k[\Delta])=0$
\end{cor}

\begin{proof}
Since every $(d-1)$ face contains a $(d-2)$ face, the result follows from above theorem and Theorem 3.5 of \cite{RWY}
\end{proof}

Necessary condition for the vanishing of $H^{d-1}_J (k[\Delta])$ is not a sufficient condition of it. We cite an example.

\begin{example}
If we consider $\Delta= 2^{\{a,b,c,d\}}-\{a,b,c,d\}$ and $\Sigma= 2^{\{a,b,c\}}\cup\{d\}$, then one can find that $H^{2}_J (k[\Delta])\neq 0$, but every 1 face of $\Delta$ contains a vertex of $\Sigma$.
\end{example} 

So for the other direction, we have the following.

\begin{thm}
Let $\Sigma \subset \Delta$ be simplicial complexes where $\dim k[\Delta]= d$ with $J=I_{\Sigma}$ (assume the situation and notation in section 2). If every $(d-1)$-face of $\Delta$ contains a vertex of $\Sigma$ and its every $(d-2)$-face is a $(d-2)$ face of $\Sigma$ and also every $(d-2)$-facet of $\Delta$ contains one vertex of $\Sigma$, then $H^{d-1}_J (k[\Delta])=0$. 
\end{thm}

\begin{proof}
For the other direction, we assume that every $(d-1)$-face contains two vertices of $\Sigma$ and every $(d-2)$-face contains one vertex of $\Sigma$ and using Theorem A of Section 2, we need to show $H^{d-1}_J (k[\Delta])_{\bold{a}}= \widetilde{H}^{d-2}(||\Star_{\Delta}(F_{+})||-||\Sigma||,||\del_{\Star_{\Delta}(F_{+})}(F_{-})||-||\Sigma||;k)= 0$ with $F_{+} \cup F_{-}\in \Delta$. Since $\Star_{\Delta}(F_{+})$ can be of dimension atmost that of $\Delta$, we can assume $F_{+}=\phi$ and above reduces that we only need to show $\widetilde{H}^{d-2}(||\Delta||-||\Sigma||,||\del_{\Delta}(F)||-||\Sigma||;k)= 0$. 

Now instead of taking simplicial homology for the pair of spaces $(||\Delta||-||\Sigma||,||\del_{\Delta}(F)||-||\Sigma||)$, by Lemma 4.7.27 of \cite{BLSWZ}, we can take the pair of spaces $(\Sd (\Delta-\Sigma),\Sd (\del_{\Delta}(F)-\Sigma))$, where $\Sd (\Delta-\Sigma)$ is the subcomplex of barycentric subdivision of $\Sd (\Delta)$ whose vertices are not the barycentre of any face of $\Sigma$. Using Theorem B of Section 2, we get that $\widetilde{H}^{d-2}(||\Delta||-||\Sigma||,||\del_{\Delta}(F)||-||\Sigma||;k)= \widetilde{H}^{d-2}(\Sd (\Delta-\Sigma),\Sd (\del_{\Delta}(F)-\Sigma))$. 

Now from above Lemma 3.1 and Remark 1, we have shown that we can remove every $(d-1)$-face and $(d-2)$-face of $\Sd (\Delta-\Sigma)$ along with those of $\Sd (\del_{\Delta}(F)-\Sigma)$ and this leads to the desired result.
\end{proof}

{\textbf{Acknowledgement:}}\newline

For example 3.1, I would like to thank the referee of the journal where I submitted this work previously.








\end{document}